\documentclass{amsart}

\usepackage[colorlinks]{hyperref}
\hypersetup{
    linkcolor=red, %[red]
    citecolor=blue, %[green]
}

\usepackage{multirow}
\usepackage{graphicx}
\newtheorem{theorem} {Theorem}[section]
\newtheorem{corollary}[theorem]{Corollary}
\newtheorem{lemma}[theorem]{Lemma}
\newtheorem{proposition}[theorem]{Proposition}
\newtheorem{question}[theorem]{Question}

\theoremstyle{definition}
\newtheorem{example}[theorem]{Example}
\newtheorem{definition}[theorem]{Definition}

\theoremstyle{remark}
\newtheorem{remark}[theorem]{Remark}
\newcommand{\conv}  {\operatorname{conv} }
\newcommand{\vol}  {\operatorname{vol} }
\newcommand{\Z}{\mathbb{Z}}
\newcommand{\R}{\mathbb{R}}
\newcommand{\N}{\mathbb{N}}
\newcommand{\size}{\text{size}}
\newcommand{\width}{\text{width}}
\newcommand{\lcm}{\operatorname{lcm}}
\newcommand{\ver}{\operatorname{vert}}
\newcommand{\dist}{\operatorname{dist}}

\begin{document}

\title{The finiteness threshold width of lattice polytopes}


%    author one information
\author[M. Blanco]{M\'onica Blanco}
\author[C. Haase]{Christian Haase}
\author[J. Hofmann]{Jan Hofmann}
\author[F. Santos]{Francisco Santos}

\address
[F.~Santos and M.~Blanco]
{
Departamento de Matem\'aticas, Estad\'istica y Computaci\'on,
Universidad de Cantabria,
39005 Santander, Spain
}
\email{francisco.santos@unican.es, m.blanco.math@gmail.com}

\address[C.~Haase and J.~Hofmann]{%
  Mathematik \\
  FU Berlin \\
  14195 Berlin \\
  Germany}
\email{haase@math.fu-berlin.de, math@hofmann-jan.de}%janhofmann@math.fu-berlin.de}

\thanks{ Supported by: grants MTM2014-54207-P, MTM2017-83750-P (M.~Blanco
  and F.~Santos), and BES-2012-058920 (M.~Blanco) of the Spanish
  Ministry of Science; the Einstein Foundation Berlin under grant EVF-2015-230 (F.~Santos); the Berlin Mathematical School (J.~Hofmann).\\
  \null \ \\}

\date{}

%----------------------------------------------------------------
% -------------- ABSTRACT ----------------------------------------
%----------------------------------------------------------------

\begin{abstract}
In each dimension $d$ there is a constant $w^\infty(d)\in \N$
such that for every $n\in \N$ all but finitely many lattice $d$-polytopes with
$n$ lattice points have lattice width at most $w^\infty(d)$. We
call $w^\infty(d)$ the \emph{finiteness threshold width} in dimension $d$ and show that $d-2
\le w^\infty(d)\le O^*\left( d^{4/3}\right)$.

Blanco and Santos determined the value $w^\infty(3)=1$. Here, we
establish $w^\infty(4)=2$. This implies, in particular, that there are only finitely many
empty $4$-simplices of width larger than two. (This last statement was claimed by Barile et al. in \emph{Proc. Am. Math. Soc.} (2011), but we have found a gap in their proof).

Our main tool is the study of $d$-dimensional lifts of hollow $(d-1)$-polytopes.

\end{abstract}

\maketitle
%\setcounter{tocdepth}{1}
%\tableofcontents

%%%%%%%%%%%%%%%%%%%%%%%%%%%%%%%%%%%
%%%%%%%				SECTION 1 			 %%%%%%%%%
%%%%%%%%%%%%%%%%%%%%%%%%%%%%%%%%%%%

\section{Introduction}

\emph{Lattice polytopes} are convex polytopes with vertices in
$\Z^d$ (or in any other lattice). They appear in combinatorics, algebraic geometry, symplectic geometry, optimization, or
statistics and have applications to mathematical physics in string theory.
In particular, 
considerable effort has gone into several classification projects for
several classes of them, with motivation stemming from different
sources. For example:
\begin{itemize}
\item A monumental task and now a shining example is the
classification of reflexive polytopes up to dimension $4$ by Kreuzer 
and Skarke~\cite{KS00}, the data for these and other Calabi-Yau manifolds
can be found online under \url{http://hep.itp.tuwien.ac.at/~kreuzer/CY.html}.

\item Smooth reflexive polytopes were classified up to dimension $8$ by 
{{\O}}bro~\cite{Ob07} and in dimension $9$ by Lorenz and Paffenholz~\cite{LoPa} (see also \url{https://polymake.org/polytopes/paffenholz/www/fano.html}). This classification 
led to new discoveries about smooth reflexive
polytopes in arbitrary dimension and hereby helped solving long-open problems~\cite{AJP14,LoNi15,NP11,OSY12}. 

\item Lattice polytopes with a single lattice point in their interior
(assumed to be the origin) are important in algebraic geometry. They
correspond to projective toric varieties with at most \emph{canonical
 singularities}, which is why they are called \emph{canonical
 polytopes}. Canonical polytopes all of whose boundary lattice points
are vertices are called \emph{terminal}. Canonical $3$-dimensional
lattice polytopes were fully enumerated by Kasprzyk~\cite{Kas08}. The
data for this and a lot more can be found in the graded ring database
(\url{http://www.grdb.co.uk}).

\item A classification especially useful for us is that of 
\emph{hollow} polytopes, by which we mean lattice polytopes
without interior lattice points. In dimension two their list consists of
the polygons of width one plus the second dilation of a unimodular triangle. 
In dimension three the full classification has recently been completed
by Averkov et al. \cite{AverkovWagnerWeismantel}
and~\cite{AKW15}. See Section~\ref{sec:dim4} for details.

\item We call \emph{empty} a (necessarily hollow) lattice polytope with no lattice point apart from its vertices.
Empty simplices are of special interest, since they are the building blocks into which every lattice polytope can be decomposed, and since they correspond to terminal quotient singularities in algebraic geometry. 
Their classification in dimension three is by now classical~\cite{White}. Their classification in dimension four has been completed recently~\cite{IglesiasSantos2}, after
efforts coming both from algebraic geometry~\cite{BarileBernardiBorisovKantor,Bober09,MMM88,Sankaran90}
and discrete geometry~\cite{HaaseZiegler}. See Remark~\ref{rem:BBBK} for more details.
\end{itemize}

All these classifications are modulo \emph{unimodular equivalence}, sometimes called $\Z$-isomorphism. We say that two lattice polytopes are \emph{unimodularly equivalent}, or just \emph{equivalent}, if there is
a lattice-preserving affine isomorphism mapping them onto each other.
\medskip

From the point of view of discrete geometry alone, it seems natural to
classify, or enumerate, \emph{all} lattice polytopes of a given
dimension and with a certain number of lattice points. We call the
latter the \emph{size} of a lattice polytope. 
In dimension $1$ this is trivial, since the unique lattice segment of size $n$ is that of length $n-1$. In dimension $2$ it is also easy, at least from the theoretical point of view: Pick's Theorem implies that lattice polygons of size $n$ have area smaller than $n$, which in turn implies that they can be unimodularly transformed to fit in $[0,2n]^2$. Hence, there are finitely many of them and they can in principle be enumerated by brute force.
%For example, Castryck~\cite{castryck2012movingout} has enumerated all lattice polygons with up to $30$ interior lattice points. 
%
However, in dimension $3$ and higher the task is a-priori undoable, since the number is infinite already for the smallest possible case, that of \emph{empty tetrahedra} (that is, lattice $3$-polytopes of size $4$). Indeed,
the following infinite family of so-called \emph{Reeve tetrahedra} was
described more than 60 years ago~\cite{Re57}:
\[
T_r:=\conv\left\{
\begin{pmatrix}0\\0\\0\\ \end{pmatrix},
\begin{pmatrix}1\\0\\0\\ \end{pmatrix},
\begin{pmatrix}0\\1\\0\\ \end{pmatrix},
\begin{pmatrix}1\\1\\r\\ \end{pmatrix}
\right\}.
\]

Still, Blanco and Santos~\cite{5points} found a way of making sense of
the question in dimension $3$. They proved that, for each $n$, all
but finitely many lattice $3$-polytopes of size $n$ have width one. They also
enumerated lattice polytopes of width larger than one and of sizes up
to eleven~\cite{5points,6points,quasiminimals}. 

Here, the width of a lattice polytope $P$ \emph{with respect to
 a linear functional} $\ell \in (\R^{d})^*$ is defined as
\[
\width_{\ell}(P) := \max_{ p,q\in P} \left| \ell\cdot p- \ell\cdot q
\right| \,,
\]
and the \emph{lattice width}, or simply \emph{width}, of the polytope $P$ is the minimum such
$\width_\ell(P)$ where $\ell$ ranges over non-zero integer
functionals:
\[
\width(P):=\min_{\ell \in(\Z^{d})^* \setminus \{0\}} \width_{\ell}(P).
\] 
For example, $P$ has width one if and only if it lies between two
consecutive lattice hyperplanes.
\medskip

The starting point in this paper is the observation that the finiteness 
result of Blanco and Santos generalizes as follows:

\begin{theorem}
\label{thm:ftw}
For each dimension $d$ there is a constant $w\in \N$ such that
for every $n\in \N$ the number of lattice $d$-polytopes of size $n$ and width larger
than $w$ is finite.
Moreover, the minimal such constant $w^\infty(d)$ satisfies
\begin{equation}
\label{eqn:w_E}
d-2 \quad \le \quad w^\infty(d) \quad \le \quad  O^*\big( d^{4/3}\big),
\end{equation}
where the notation $O^*$ means that a polylogarithmic factor is omitted.
\end{theorem}

\begin{proof}
Corollary~\ref{coro:infinite_bipyramid}  (see  Theorem~\ref{thm:summaryofbounds}(5)) states that $w^\infty(d) \ge w_H(d-2)$, where $w_H(s)$ denotes the maximum lattice width of hollow lattice $s$-polytopes. Since the $s$-th dilation of a unimodular $s$-simplex is hollow, one has $w_H(s) \ge s$, hence \[w^\infty(d) \ge w_H(d-2) \ge d-2.\]

For the upper bound, first observe that hollow polytopes of dimension $d$ have a global bound of $O^*(d^{4/3})$ for their width~\cite{Banaszczyk_etal, Rudelson2000}. 
That is, $w_H(d)\le O^*(d^{4/3})$. It hence suffices to show that $w^\infty(d) \le w_H(d)$. For this we use:

{\bf Claim:} \emph{For each $d$ and $n$ there are only finitely many non-hollow lattice $d$-polytopes of size $n$}. This follows from the combination of the following two results: 
Hensley~\cite[Thm.~3.6]{Hensley} showed that there is a bound on the
volume of non-hollow lattice $d$-polytopes with a given number $k$ of interior points; taking the maximum of these 
bounds for $k\in\{1,\dots,n-d-1\}$ provides a bound for the volume of non-hollow $d$-polytopes of size $n$. 
Lagarias and Ziegler~\cite[Thm.~2]{LagariasZiegler} proved finiteness of
the number of equivalence classes of lattice $d$-polytopes with a
bound on their volume.

Since $d$-polytopes of width $>w_H(d)$ are necessarily non-hollow, the claim implies that for each $n$ there are finitely many $d$-polytopes of size $n$ and of width $>w_H(d)$; that is, $w^\infty(d) \le w_H(d)$.
(In Theorem~\ref{thm:summaryofbounds}(3) below, we tighten this slightly to $w^\infty(d) \le w_{H}(d-1)$).
\end{proof}

\begin{definition}
For each $d\in \N$ we call \emph{finiteness threshold width} in dimension $d$ the minimum constant $w^\infty(d)\in \N$ such that for every $n\in \N$ the number of lattice $d$-polytopes of size $n$ and width larger than $w^\infty(d)$ is finite.
\end{definition}

For instance, $w^\infty(1)=w^\infty(2)=0$ since, as said above, there are only finitely many lattice $d$-polytopes of each size. Blanco and Santos'
aforementioned result states that $w^\infty(3)=1$. Our main result is the exact value of 
$w^\infty(4)$:

\begin{theorem}[Corollary~\ref{coro:main}]
\label{coro:main-intro}
$w^\infty(4)=2$. That is, for each $n \ge 5$ there are only finitely
many lattice $4$-polytopes of size $n$ and width greater than $2$.
\end{theorem}

This implies the following result:

\begin{corollary}
\label{coro:BBBK}
There are infinitely many empty $4$-simplices of width two but only finitely many
of larger width.
\end{corollary}

\begin{proof}
Haase and Ziegler~\cite[Proposition 6]{HaaseZiegler} 
found infinitely many empty $4$-simplices of width~$2$.
$w^\infty(4)=2$ implies there are only finitely many of larger width.
\end{proof}

\begin{remark}
\label{rem:BBBK}
The second part of Corollary~\ref{coro:BBBK} is the main result in Barile et
al.~\cite{BarileBernardiBorisovKantor}, but we have found out that
the proof given in that paper is incomplete.
Indeed, the authors use a classification of infinite families of empty
$4$-simplices of width $>1$ that had been conjectured 
to be complete 
by Mori et al.~\cite{MMM88} and proved by Sankaran~\cite{Sankaran90}
and Bover~\cite{Bober09}, \emph{
for simplices whose determinant (i.e., their normalized volume)
is a prime number}. 
But when the determinant is not prime other infinite families do arise, such as the
following explicit example: the empty $4$-simplices with vertices $e_1$,
$e_2$, $e_3$, $e_4$ and $(2, N/2-1, a, N/2-a)$, where the determinant $N$ is a
multiple of $4$ and coprime with $a$. 
As a conclusion,
the proof of Corollary~\ref{coro:BBBK} given in~\cite{BarileBernardiBorisovKantor}
is valid only for simplices of prime determinant.

We thank O.~Iglesias for the computations
leading to finding this (and other) families and we thank the authors of~\cite{BarileBernardiBorisovKantor}
for acknowledging (private communication) their mistake and for
helpful discussions about the extent of it.

After the present paper was completed, a new proof of Corollary~\ref{coro:BBBK} has been obtained by Iglesias and Santos, which gives the following more explicit information: there are exactly $179$ empty $4$-simplices of width larger than two, all of width three except for one of width four~\cite{IglesiasSantos}. Furthermore,~\cite{IglesiasSantos2} contains the full classification of empty $4$-simplices, including the additional infinite families of width two that arise for nonprime determinant.
\end{remark}

Our bounds on $w^\infty(d)$ come from relating it to the maximum widths of hollow and/or empty $d$-polytopes.
As already mentioned, a lattice polytope is \emph{hollow} if there is no lattice
point in its interior and \emph{empty} if its vertices are the only lattice
points it contains.

\begin{definition}
\label{defi:w_H}
We denote $w_H(d)$ and $w_E(d)$ the maximum widths
of hollow and empty $d$-polytopes, respectively.
\end{definition}

Finiteness of $w_H(d)$ (and hence of $w_E(d)$) is usually called the ``flatness theorem'', dating back to Khinchine (1948); see, e.g.,~\cite{KL88}. The current best upper bound of $w_H(d) \le O(d^{4/3} \log^ad)$ for some constant $a$ (used in the proof of Theorem~\ref{thm:ftw}) is by Rudelson~\cite{Rudelson2000}, building on work by Banaszczyk et.~al~\cite{Banaszczyk_etal}.
As for lower bounds, $w_H(d) \ge d$ follows from hollowness of the $d$-th dilation of a unimodular $d$-simplex, while 
$w_E(d) \ge 2\lfloor{d/2}\rfloor-1$ was proved by Seb{\H{o}}~\cite{Sebo} by slightly modifying this same dilated $d$-simplex to make it empty. 
\smallskip

Along the paper, we prove the following properties and bounds of $w^\infty(d,n)$ and $w^\infty(d)$,
where $w^\infty(d,n)$ is the stratification of the threshold width in terms of size. That is,
$w^\infty(d,n)\in \N \cup \{\infty\}$ is the minimal width $W\ge0$ such that there exist only
finitely many lattice $d$-polytopes of size $n$ and width $>W$. Clearly, $w^\infty(d) = \max_{n\in \N} \; w^\infty(d,n)$ and, in particular, each $w^\infty(d,n)$ is finite.
\smallskip

\begin{theorem}
\label{thm:summaryofbounds}\ 
\begin{enumerate}
\item $w^\infty(d,n)\le w^\infty(d,n+1)$ for all $d,n$. \hfill
  (Proposition~\ref{proposition:monotone_n})
\item $w^\infty(d)\le w^\infty(d+1)$ for all $d$. \hfill
  (Proposition~\ref{proposition:monotone_d})
\item $w^\infty(d)\le w_H(d-1)$. \hfill
  (Lemma~\ref{lemma:finite-nonprojecting})
\item $w_E(d-1) \le w^\infty(d)$ for $d\ge 3$. \hfill
  (Corollary~\ref{coro:infinite_nonsimplex})
\item $w_H(d-2)\le w^\infty(d)$. \hfill
  (Corollary~\ref{coro:infinite_bipyramid})
\end{enumerate}
\end{theorem}

\begin{remark}
\label{rem:known_values}
None of the inequalities $w_H(d-2) \le w^\infty(d) \le w_H(d-1)$ or
{$w_E(d-1) \le  w^\infty(d)$} (for $d \ge 3$) is sharp, as the following
table of known values shows.

\[
\begin{array}{c|cc|c|c}
\multirow{2}{*}{$d$} & \multicolumn{2}{|c|}{\text{lower bounds}} & \multirow{2}{*}{$w^\infty(d)$} & \text{upper bound}\\
 & w_E(d-1)  &w_H(d-2)&  & w_H(d-1) \\
\hline
1 &-               &-           &0                &-\\
2 & 1             &-            &0               &1\\
3 & 1             &1           &1           	  &2	\\
4 & 1             &2	          &2           &3\\
5 & \ge 4        &3           &\ge 4&\ge 4\\
\end{array}
\medskip
\]
The values of $w^\infty(d)$, $d=1,2,3,4$, have been discussed
above.
For the rest:
\begin{itemize}
\item In dimension $1$, the unique hollow lattice segment is equivalent to $[0,1]$, and then $w_E(1)=w_H(1)=1$.

\item In dimension $2$, the second dilation of a unimodular triangle is the only hollow lattice polygon of width larger than one (see, e.g., \cite{Treutlein}). Hence $w_H(2)=2$ and, since this polygon is not empty, $w_E(2)=1$. 

\item In dimension $3$, Howe (\protect{\cite[Thm.~1.3]{Scarf}}) proved that $w_E(3)=1$. 
For $w_H(3)$, Averkov et
al. (\protect{\cite[Theorem~2.2]{AverkovWagnerWeismantel}}
and~\protect{\cite[Theorem~1]{AKW15}}) have classified all hollow $3$-polytopes and their maximum width is three (see more details in Lemma \ref{lemma:w_H(3)}), so $w_H(3)=3$.

\item In dimension $4$, Haase and Ziegler~\cite{HaaseZiegler} showed $w_E(4)\ge 4$, which implies $w^\infty(5)\ge 4$ by part (4) of Theorem~\ref{thm:summaryofbounds}. 
\end{itemize}
\end{remark}

\medskip

The structure of the paper is as follows. The monotonicity properties 
stated in parts (1) and (2) of Theorem~\ref{thm:summaryofbounds}
are proved at the beginning of 
Section~\ref{sec:threshold-lifts}.
We then prove the upper bound $w^\infty(d) \le w_H(d-1)$
(Lemma~\ref{lemma:finite-nonprojecting}) from the following statement, which
 combines results of Hensley~\cite{Hensley},
 Lagarias--Ziegler~\cite{LagariasZiegler} and
 Nill--Ziegler~\cite{NillZiegler}: all but finitely many
lattice $d$-polytopes of bounded size are hollow and project to hollow
$(d-1)$-polytopes.
This fact implies that in order to find an infinite family of lattice
$d$-polytopes of bounded size we can focus on lifts (see
Definition~\ref{defi:lift}) of hollow polytopes of one dimension
less. The fact that all but finitely many lifts of a lattice $(d-1)$-polytope are $d$-dimensional and have the same width (Theorem~\ref{theorem:same_width}) then implies that in order to 
decided whether a lattice polytope has infinitely many lifts of bounded size 
it is enough to look at \emph{tight} lifts, which are inclusion-minimal lifts of a
polytope (see
Definition~\ref{defi:tight} and Corollary~\ref{coro:equiv_statements}).

In Section~\ref{sec:infinitely-hollow-lifts} we prove sufficient properties
for hollow $(d-1)$-polytopes to have infinitely many lifts of
bounded size. In particular, we prove the existence of such
hollow $(d-1)$-polytopes of widths $w_E(d-1)$ and $w_H(d-2)$, which
provides the lower bounds $w_E(d-1) \le w^\infty(d)$
(Corollary~\ref{coro:infinite_nonsimplex}) and $w_H(d-2) \le
w^\infty(d)$ (Corollary~\ref{coro:infinite_bipyramid}).
Moreover, we get the following characterization
of the finiteness threshold width:

\begin{theorem}[Theorem~\ref{thm:one_hollow_Q} and Corollary~\ref{coro:equiv_statements}]
\label{thm:one_hollow_Q-intro}
For all $d\ge 3$, 
$w^\infty(d)$ equals the maximum width of a hollow lattice $(d-1)$-polytope $Q$ for which there are
infinitely many (equivalence classes of) lattice $d$-polytopes $P$ of bounded size projecting to $Q$.
\end{theorem}

One direction of the theorem is easy, since a $Q$ as in the statement has all but finitely many of its lifts of the same width as $Q$ (Theorem~\ref{theorem:same_width}). The other is less obvious since $w^\infty(d)$
might a priori be achieved by the existence of 
infinitely many hollow $(d-1)$-polytopes $Q$,
each with finitely many lifts of size $n$.

\begin{example}
\label{exm:threshold_via_lifts}
In dimension $3$, the infinite family of Reeve tetrahedra are lifts of size $4$ of a unit square, which is a hollow polygon of width one. On the other hand, the unique hollow polygon of width larger than one is the second dilation of the unimodular triangle, which has only finitely many lifts of bounded size (see~\cite[Corollary 22]{5points}). Hence $w^\infty(3)=1$.

In dimension $4$, observe that $w^\infty(4)\ge 2$ follows from the fact that the following hollow $3$-polytope of width two can be lifted to infinitely many empty $4$-simplices (Haase and Ziegler~\cite[Proposition 6]{HaaseZiegler}):
\[
Q=\conv\left\{
\begin{pmatrix}0\\0\\0\\ \end{pmatrix},
\begin{pmatrix}1\\0\\0\\ \end{pmatrix},
\begin{pmatrix}0\\1\\0\\ \end{pmatrix},
\begin{pmatrix}0\\0\\1\\ \end{pmatrix},
\begin{pmatrix}2\\2\\3\\ \end{pmatrix}
\right\}.
\]
\end{example}

Sections~\ref{sec:finitely-hollow-lifts} and~\ref{sec:dim4} are aimed at proving our main result $w^\infty(4)=2$ (Theorem~\ref{coro:main-intro}).
By Theorem~\ref{thm:one_hollow_Q-intro} and Example~\ref{exm:threshold_via_lifts}, it suffices to show that each hollow $3$-polytope of width larger than two has finitely many $4$-dimensional lifts of bounded size. For this we first prove sufficient conditions for lattice polytopes (in arbitrary dimension) to
have only finitely many lifts of bounded size (Section~\ref{sec:finitely-hollow-lifts}).
Subsequently in Section~\ref{sec:dim4} we apply this to the list of hollow $3$-polytopes of width larger than two. This list, containing only five polytopes, is derived from the 
classification of \emph{maximal} hollow $3$-polytopes by Averkov et al.~(\protect{\cite[Theorem 2.2]{AverkovWagnerWeismantel}} and~\protect{\cite[Theorem 1]{AKW15}}). 

\medskip

In light of these results, we ask the following questions.

\begin{question}
Besides the monotonicity in parts $(1)$ and $(2)$ of Theorem~\ref{thm:summaryofbounds},
does $w^\infty(d,n)\le w^\infty(d+1,n+1)$ always hold? 
The case $w^\infty(d,d+1) \le w^\infty(d+1,d+2)$ follows from
\protect{\cite[Proposition~1]{HaaseZiegler}}: every empty $d$-simplex
is a facet of infinitely many empty $(d+1)$-simplices of at least the
same width. 
\end{question}

\begin{question}
For all known values ($d\le 4$) we have $w^\infty(d) =
w^\infty(d,d+1)$. That is, the finiteness threshold width for all lattice
$d$-polytopes is determined by empty $d$-simplices. Does this hold for arbitrary $d$?
\end{question}

\subsection*{Acknowledgement:} We thank Benjamin Nill and Gennadiy Averkov for helpful discussions,
in particular for pointing us to references~\cite{AverkovWagnerWeismantel,AKW15}. 
An anonymous referee made several useful comments, in particular giving us the current proof of Theorem~\ref{theorem:same_width}, much simpler than the one we originally had.

%%%%%%%%%%%%%%%%%%%%%%%%%%%%%%%%%%%
%%%%%%%				SECTION 2 			 %%%%%%%%%
%%%%%%%%%%%%%%%%%%%%%%%%%%%%%%%%%%%
\section{Finiteness threshold width and lifts of hollow polytopes}
\label{sec:threshold-lifts}

\subsection*{Monotonicity of the finiteness threshold widths}

Parts (1) and (2) of Theorem~\ref{thm:summaryofbounds} have the
following proofs:

\begin{proposition}
\label{proposition:monotone_n}
$w^\infty(d,n)\le w^\infty(d,n+1) $ for all $n\ge d+1$.
\end{proposition}
\begin{proof}
Let $W\in \N$ be such that there exists an infinite family
$\{P_i\}_{i\in \N}$ of lattice $d$-polytopes of size $n$ and width $W$.
We are going to show that for each $P_i$ there is a $P'_i$ of size $n+1$ and width
$W$ containing $P_i$. To prove this, let $\ell_i$ be an integer
functional giving width $W$ to $P_i$, and assume without loss of
generality that $\ell_i(P_i)=[0,W]$. Taking any point $q_i \in \Z^d
\cap \ell_i^{-1}[0,W] \setminus P_i$ we easily get a
$Q_i=\conv(P_i\cup\{q_i\})$ of width $W$ and properly containing
$P_i$ (see Figure~\ref{fig:monotonicity}). If $Q_i \setminus P_i$ has more than one lattice point, remove
them one by one until only one remains (which can always be done;
simply choose a vertex $v$ of $Q_i$ not in $P_i$ and replace $Q_i$ with
the convex hull of $(Q_i \cap \Z^d) \setminus \{v\}$; then iterate).

\begin{figure}[htb]
\centerline{\includegraphics[scale=.8]{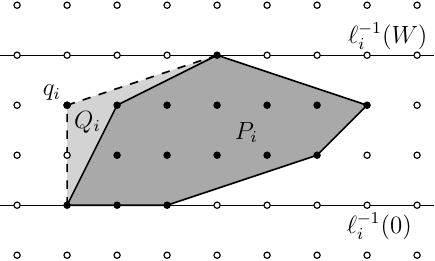}}
\caption{The setting of the proof of Proposition~\ref{proposition:monotone_n}.}
\label{fig:monotonicity}
\end{figure}

That implies the lemma except for the fact that 
different polytopes $P_i$ and $P_j$ may produce isomorphic
$P'_i$ and $P'_j$, so it is not obvious that $\{P_i'\}_{i\in \N}$ is
an infinite family.
But each element of $\{P_i'\}_{i\in \N}$ can only correspond to \emph{at most} $n+1$ 
elements from $\{P_i\}_{i\in \N}$ 
(because $P_i$ is recovered from $P'_i$ by removing one of its $n+1$ lattice points),
so the proof is complete.
\end{proof}

\begin{proposition}
\label{proposition:monotone_d}
$w^{\infty}(d)\le w^{\infty}(d+1)$, for all $d$.
\end{proposition}

\begin{proof}
Let $W\in \N$ be such that, for some $n\in \N$,
there is an infinite family $\{P_i\}_{i\in \N}$ of lattice $d$-polytopes of
size $n$ and width $W$. Then, ${\mathcal P}=\{P_i\times [0,W]\}_{i\in \N}$ is a 
family of $(d+1)$-polytopes of size $n\cdot (W+1)$ and width $W$. 
A priori two different $P_i$'s can give isomorphic polytopes in ${\mathcal P}$,
but each polytope in ${\mathcal P}$ can correspond to only finitely many $P_i$'s
since $P_i$ is the projection of $P_i\times [0,W]$ along the direction of an edge.
Hence ${\mathcal P}$ is infinite and $w^\infty(d+1) \ge W$.
\end{proof}

The following lemma proves part (3) of Theorem~\ref{thm:summaryofbounds}:

\begin{lemma}
\label{lemma:finite-nonprojecting}
Let $d< n\in \N$.
All but finitely many lattice $d$-polytopes of size bounded by $n$ 
are hollow and
admit a projection to some hollow lattice $(d-1)$-polytope.
In particular, $w^\infty(d) \le w_H(d-1)$ for all $d$.
\end{lemma}

\begin{proof}
As argued in the proof of Theorem~\ref{thm:ftw}, the number of non-hollow lattice $d$-polytopes of size bounded by $n$ is finite.
Hence, all but finitely many lattice
$d$-polytopes of size bounded by $n$ are hollow.

On the other hand, Nill and Ziegler~\cite[Corollary~1.7]{NillZiegler} proved that
all but finitely many hollow $d$-polytopes admit a projection to a
hollow $(d-1)$-polytope. 
And these have width at most that of their projection, which is $\le w_H(d-1)$.
\end{proof}

\subsection*{Finiteness threshold width via polytopes with infinitely many lifts of bounded size} 
\ 

\begin{definition}
\label{defi:lift}
We say that a lattice polytope $P\subset \R^d$ is a \emph{lift} of a lattice $(d-1)$-polytope $Q\subset \R^{d-1}$ if there is a lattice projection $\pi:\R^d\to \R^{d-1}$ with $\pi(P)=Q$. Without loss of generality, we will typically assume $\pi$ to be the map that forgets the last coordinate. 

Two lifts $P_1$ and $P_2$, with projections $\pi_1:P_1 \to Q$ and $\pi_2:P_2 \to Q$ are \emph{equivalent} if there is a unimodular equivalence $f:P_1\to P_2$ with $\pi_2\circ f = \pi_1$. That is, if for each $p\in \Z^d$, $f(p) \in \pi_2^{-1}(\pi_1(p))$ (the equivalence maps a point in the fiber of $p$ under $\pi_1$, to a point in the fiber of $p$ under $\pi_2$). See Figure~\ref{fig:lift_equiv} for examples of equivalent and non-equivalent lifts.

We say that ``$Q$ has finitely many lifts of bounded size'' if for every $n\in \N$ there are finitely many equivalence classes of lifts of $Q$ of size $n$. Accordingly, ``$Q$ has infinitely many lifts of bounded size'' means that there is an $n\in \N$ for which there are infinitely many equivalence classes of lifts of $Q$. 
\end{definition}

\begin{figure}[h]
\centerline{\includegraphics[scale=.7]{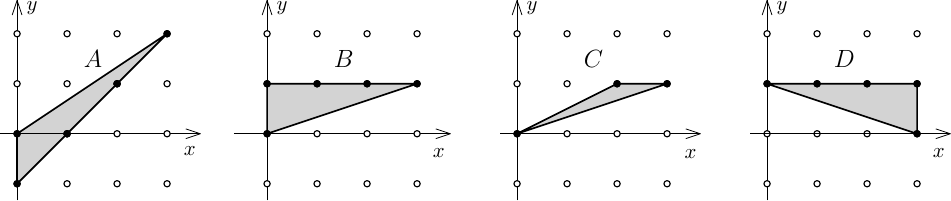}}
\caption{Polytopes $A,B,C,D \subset \R^2$ are lifts of $[0,3]\subset \R$ under projection $\pi(x,y)=x$. Only $A$ and $B$ are equivalent lifts. $D$ is equivalent to $A$ and $B$ as a lattice polytope, but not as a lift of $[0,3]$ under $\pi$.}
\label{fig:lift_equiv}
\end{figure}

\begin{remark}
\label{rm:equivalent_lift}
Saying that ``$Q\subset\R^{d-1}$ has infinitely many lifts of bounded size'' is equivalent to saying that ``there are infinitely many (equivalence classes of) lattice polytopes $P\subset\R^d$ of bounded size that have a lattice projection to $Q$''. The implication from right to left is trivial, and the implication from left to right follows from the fact that once $P$ is fixed there is a finite number of integer affine projections $P\to Q$ (an overestimate is $q^p$, where $p$ and $q$ are the numbers of lattice points in $P$ and $Q$, respectively; $q^{d+1}$ is also an upper bound, since an affine map is determined by the image of an affine basis).
\end{remark}

Our interest on these concepts comes from the following fact: for all $d\ge 3$, 
$w^\infty(d)$ is \emph{at least} the maximum width of a lattice $(d-1)$-polytope $Q$ that admits infinitely many lifts of bounded size (see Corollary~\ref{coro:one_hollow_Q}). 
For its proof we need a couple of technical results about the dimension and the width of the lifts of a polytope.

A lift of $Q$ may have the same dimension as $Q$ and still not be unimodularly equivalent to it. For example, the segment $[0,k]$ in $\R^1$ can be lifted to the primitive segment $\conv\{(0,0), (k,1)\}$. However, the number of different such lifts of $Q$ is finite, modulo the equivalence relation in Definition~\ref{defi:lift}:

\begin{lemma}
\label{lemma:finite_same_dim}
A $(d-1)$-dimensional polytope $Q$ has only finitely many $(d-1)$-dimensional lifts.
\end{lemma}

\begin{proof}
Every $(d-1)$-dimensional lift $P$ of $Q$ can be described as follows: there is an affine map $f:\R^{d-1}\to \R$ with
\[
P=\conv\{ (v,f(v)): v \text{ is a vertex of $Q$}\},
\]
and such that $f$ is integer in all vertices of $Q$. Assuming, without loss of generality, that $f$ is linear and the origin is a vertex of $Q$, this implies $f\in \Lambda(Q)^*$, where $\Lambda(Q)$ is the lattice spanned by the vertices of $Q$. Two such functionals give equivalent lifts if, and only if, they are in the same class modulo $(\Z^{d-1})^*$. Thus, the number of different lifts equals the index of $(\Z^{d-1})^*$ in $\Lambda(Q)^*$.
\end{proof}

\begin{theorem}
\label{theorem:same_width}
 Let $Q \subset \R^{d-1}$ be a lattice $(d-1)$-polytope of width
 $W$. Then all lifts $P \subset \R^d$ of $Q$
 have width $\le W$. All but finitely many of them have width $= W$.
\end{theorem}

\begin{proof}
	The first part of the statement is clear, since projecting cannot decrease width.
	It remains to show that only finitely many lifts of $Q$ have width strictly smaller than $\width(Q)$, 
	which follows from the next claim: if a lift $P$ of $Q$ has width smaller than $\width(Q)$ then
	\[ \vol(P) \le \width(P) \vol(Q) < \width(Q) \vol(Q).\]
	Finiteness of the volume of $P$ implies finitely many possibilities for $P$.
	
To prove the volume bound, let $P$ be a lift of $Q$ of width $W':=\width(P)<W$, and let $\ell=(\ell_1,\dots,\ell_d) \in (\Z^d)^*\setminus\{0\}$ be an integer functional attaining $\width_{\ell}(P) = W'$. 
We have that $\ell_d \neq 0$ since $\ell_d = 0$ implies $\width_{\ell}(P) = \width_{(\ell_1,\dots,\ell_{d-1})}(Q) \ge W$. 
Assume without loss of generality that $\ell(P) = [0, W']$. 
Then we have that $P$ is contained in
\[
(Q\times \R) \cap \ell^{-1}([0, W']),
\]
which is a slanted prism projecting to $Q$ and with every fiber of length $W'/|\ell_d|$. (For the latter, observe that each fiber is a segment with endpoints $(x,a)$, $(x,b)$ for some $x \in \R^{d-1}$ and with $W' = |\ell(x,a) - \ell (x,b)| = |\ell_d| \, |b-a|$).
Hence, the Euclidean volume of this slanted prism is
\[
\frac{W'}{|\ell_d|} \cdot \vol(Q)
\]
and, since $P$ is contained in it, we get that
\[
\vol(P)\le
\frac{W'}{|\ell_d|} \cdot \vol(Q) 
\le  W' \cdot \vol(Q) < W\cdot \vol(Q),\]
where the middle inequality follows from $\ell_d$ being a nonzero integer.

We thank an anonymous referee for this proof, significantly simpler than the one
we originally had.
\end{proof}

This gives us:

\begin{corollary}
\label{coro:one_hollow_Q}
For all $d\ge 3$, 
$w^\infty(d)$ is at least the maximum width of a lattice $(d-1)$-polytope that admits infinitely many lifts of bounded size.
\end{corollary}
\begin{proof}
Let $Q$ be a lattice $(d-1)$-polytope of width $W$ with infinitely many lifts of bounded size. Since all but finitely many of its lifts are $d$-dimensional (Lemma~\ref{lemma:finite_same_dim}) and of width $W$ (Theorem~\ref{theorem:same_width}), we can find among them an infinite family of lattice $d$-polytopes of bounded size and width $W$. Hence $w^\infty(d)\ge W$.
\end{proof}

Less obvious is the converse, that we prove in Theorem~\ref{thm:one_hollow_Q}.

\subsection*{Tight lifts}
We finish this section showing that in order to decide whether a given $Q$ has infinitely many lifts of bounded size it is enough to look at
\emph{tight} lifts. This will simplify the work in the rest of the paper:

\begin{definition}
\label{defi:tight}
Let $Q \subset \R^{d-1}$ be a $(d-1)$-dimensional lattice polytope. We say that a lift $P \subset \R^d$ of $Q$ is \emph{tight} if the projection sending $P$ to $Q$ bijects their sets of vertices.
That is, if $P=\conv\{(v,h_v): v\in \ver(Q)\}$ for some $\mathbf{h}\in \Z^{\ver(Q)}$.

See Figure~\ref{fig:tight_lift} for examples of tight and not tight lifts.

\begin{figure}[htb]
\centerline{\includegraphics[scale=.7]{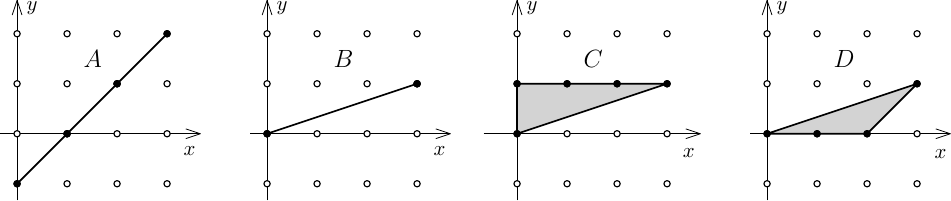}}
\caption{Polytopes $A,B,C,D \subset \R^2$ are lifts of $[0,3]\subset \R$ under the projection $\pi(x,y)=x$. $A$ and $B$ are tight lifts; $C$ and $D$ are not.}
\label{fig:tight_lift}
\end{figure}
\end{definition}

Notice that a tight lift can be $d$ or $(d-1)$-dimensional. The following lemma states that every lift of $Q$ contains a tight lift.

\begin{lemma}
\label{lemma:finite-containingP}
Let $P\subset \R^d$ be a (not necessarily full-dimensional) lift of a lattice $(d-1)$-polytope $Q$. Then, there are only finitely many lifts of $Q$ of bounded size that contain $P$.
\end{lemma}

\begin{proof}

For each $q\in Q \cap \Z^{d-1}$, pick $h_q \in \R$ such that $p_q=(q,h_q)\in P$ (these exist as $P$ projects to $Q$).

Let $P'\subset \R^d$ be any lift of $Q$ that contains $P$. Given $p'\in P' \cap \Z^d$, then $p'=(q,h')$ for some $q \in Q \cap \Z^{d-1}$
and $h'\in \Z$. Without loss of generality assume that $h'\ge h_q$ (the other case is symmetric).
Then $P'$ contains the segment $\conv\{p_q,p'\}=\{q\} \times [h_q,h']  \subset P'$, which already contains $h' - \lceil h_q \rceil +1$ lattice points (see Figure~\ref{fig:containingP}).
Since the size of $P'$ is bounded, there are finitely many possibilities for $h'$ and hence for all points of $P'$.
\medskip
\begin{figure}[htb]
\centerline{\includegraphics[scale=.8]{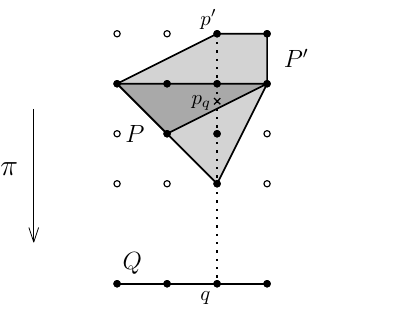}}
\caption{The segment $\pi^{-1}(q)\cap P$ has to be bounded for all $q\in Q\cap \Z^d$ in the proof of Lemma~\ref{lemma:finite-containingP}.}
\label{fig:containingP}
\end{figure}
\end{proof}
%\medskip

The results of this section lead to the following:
\begin{corollary}
\label{coro:equiv_statements}
Let $Q \subset \R^{d-1}$ be a lattice $(d-1)$-polytope. The following are equivalent:

\begin{enumerate}
\item[(1)] There are infinitely many (isomorphism classes of) lattice $d$-polytopes of bounded size projecting to $Q$.
\item[(2)] $Q$ has infinitely many lifts of bounded size.
\item[(3)] $Q$ has infinitely many lifts of bounded size and of the same width as $Q$.
\item[(4)] $Q$ has infinitely many tight lifts of bounded size.
\end{enumerate}
In any of those cases, the width of $Q$ is a lower bound for $w^\infty(d)$.
\end{corollary}

\begin{proof}
$(1) \Longrightarrow (2)$, $(2) \Longleftarrow (3)$, and $(2) \Longleftarrow (4)$ are obvious. 
For the converses:
$(2) \Longrightarrow (1)$ is Remark~\ref{rm:equivalent_lift} together with Lemma~\ref{lemma:finite_same_dim},
$(2) \Longrightarrow (3)$ follows from Theorem~\ref{theorem:same_width}, and
$(2) \Longrightarrow (4)$ comes from Lemma~\ref{lemma:finite-containingP} and the fact that any lift of a polytope contains a tight lift.
That $w^\infty(d) \ge \width(Q)$ is Corollary~\ref{coro:one_hollow_Q}.
\end{proof}
\medskip

%%%%%%%%%%%%%%%%%%%%%%%%%%%%%%%%%%%
%%%%%%%				SECTION 3 			 %%%%%%%%%
%%%%%%%%%%%%%%%%%%%%%%%%%%%%%%%%%%%

\section{Hollow polytopes with infinitely many lifts of bounded size}
\label{sec:infinitely-hollow-lifts}

By a \emph{pyramid} we mean a polytope with all but one of its vertices (called the \emph{apex}) contained in a facet (called the \emph{base}). 

\begin{lemma}
\label{lemma:infinite-hollow+(semi)empty}
Let $Q\subset \R^{d-1}$ be a lattice hollow $(d-1)$-polytope and let $v\in \ver(Q)$ be such that $Q$ is not a pyramid with apex at $v$ (That is, $Q':=\conv(\ver(Q) \setminus \{v\})$ is $(d-1)$-dimensional). Suppose that every proper face $F$ with $v\in F$ is either hollow or a pyramid with apex $v$.
Then, for every $h\in \Z\setminus \{0\}$ the $d$-dimensional tight lift $P(h):=\conv\left((Q'\times\{0\}) \cup \{(v,h)\}\right)$ of $Q$ has the following properties:
\begin{enumerate}
\item $\size(P(h)) \le \size(Q)$, with equality for infinitely many values of $h$.
\item $\width(P(h)) =\width (Q)$ for every sufficiently large $h$.
\end{enumerate}
See Figure~\ref{fig:infinitely} for an example of this layout.
\end{lemma}

\begin{figure}[htb]
\centerline{\includegraphics[scale=.9]{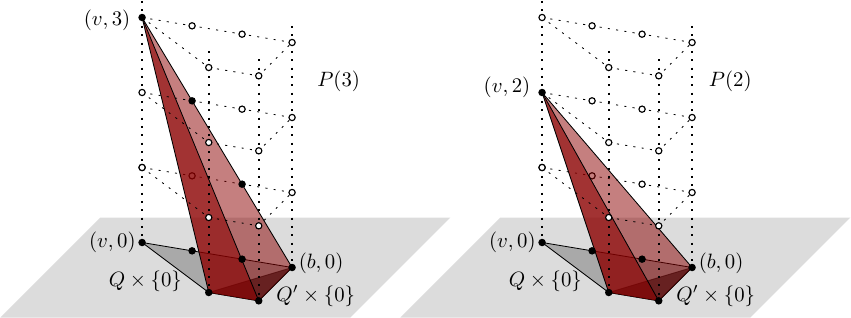}}
\caption{The setting of Lemma~\ref{lemma:infinite-hollow+(semi)empty}. The figure shows the hollow polygon $Q$ and two of its tight lifts, $P(3)$ and $P(2)$. One of the edges of $Q$ containing $v$ is empty, and the other is a $1$-dimensional  pyramid over a point $b$, with the distance from $v$ to $b$ being $3$. This implies that $P(3)$ has as many lattice points as $Q$ and $P(2)$ has strictly less lattice points.}
\label{fig:infinitely}
\end{figure}

\begin{proof}
Fix $h\in\Z\setminus \{0\}$ and let $P=P(h)$ be as in the statement.

For the first statement, let $q \in Q\cap \Z^{d-1}$. We claim that the fiber $\pi^{-1}(q)$ has at most one lattice point in $P$, with equality in many cases. For this, let $F$ be the carrier face of $q$ in $Q$ (that is, the unique face with $q \in \operatorname{relint}(F)$). Since $Q$ is hollow, $F$ is a proper face. By assumption, there are three possibilities for $F$:
\begin{itemize}
\item \emph{$F$ does not contain $v$.} Then $\pi^{-1}(F) \cap P = F\times \{0\}$. In particular, $(q,0)$ is the only lattice point of $P$ in the fiber $\pi^{-1}(q)$.
\item \emph{$v\in F$ and $F$ is hollow.} Since $q\in\operatorname{relint}(F)$, we must have $F=\{q\}=\{v\}$. In particular, $(v,h)$ is the only lattice point of $P$ in the fiber $\pi^{-1}(q)$.
\item \emph{$F$ is a pyramid with apex at $v$.} Let $F'$ be the base of the pyramid. Remember that $v$ is lifted to $(v,h)$ and every other vertex $w$ of $F'$ is lifted to $(w,0)$. In particular, the face $\pi^{-1}(F) \cap P$ of $P$ equals the affine image of $F$ under the map
$x\mapsto (x,h\cdot \dist(F',x)/\dist(F',v))$, where $\dist(F', x)$ denotes the lattice distance from $x$ to (the hyperplane spanned by) $F'$. Thus, $(q,h\cdot\dist(F',q)/\dist(F',v))$ is the only point of $P$ in the fiber $\pi^{-1}(q)$. That point will be a lattice point if (but perhaps not only if) $h$ is an integer multiple of $\dist(F',v)$.
\end{itemize}

In particular, we have $\size(P(h)) = \size(Q)$ for any $h$ that is an integer multiple of 
$
\lcm\{\dist(F',v): F \text{ face of }Q\text{ that is a pyramid with base }F'\text{ and apex }v\}.
$

The second statement follows directly from Theorem~\ref{theorem:same_width}. Indeed, the polytopes $P(h)$ are unimodularly non-isomorphic for different values of $|h|$, since their volume is proportional to $|h|$.
\end{proof}

\begin{corollary}
\label{coro:holds-empty-nonsimplex}
Let $Q$ be a hollow polytope and not a simplex.
If $Q$ is either empty or simplicial then it has infinitely many lifts of the 
same size and width of $Q$.
\end{corollary}

\begin{proof}
Since $Q$ is not a simplex, there is a vertex $v$ such that $Q$ is not a pyramid with apex at $v$.
Being empty or simplicial guarantees the conditions of Lemma~\ref{lemma:infinite-hollow+(semi)empty} for $v$ are met.
\end{proof}

These results give us a way to lower-bound $w^\infty(d)$; if a lattice $(d-1)$-polytope $Q$ is in the conditions of Lemma~\ref{lemma:infinite-hollow+(semi)empty} or the Corollary~\ref{coro:holds-empty-nonsimplex}, then $w^\infty(d)\ge \width(Q)$. 

\begin{definition}
\label{def:maximal}
A hollow lattice $d$-polytope is called \emph{hollow-maximal} if it is
maximal under inclusion of hollow lattice $d$-polytopes.

An empty lattice $d$-polytope is called \emph{empty-maximal} if it is
maximal under inclusion of empty lattice $d$-polytopes.
\end{definition}

\begin{lemma}
\label{lemma:max_hollow_or_empty}
Let $Q$ be a hollow-maximal or empty-maximal $d$-polytope, for $d\ge 2$. Then, for every vertex $v$ of $Q$ there is a lattice point $u\in Q$ that is not contained in any facet containing $v$.
\end{lemma}

\begin{proof}
Let $v$ be a vertex of $Q$ and suppose that every lattice point of $Q$ is in a facet containing $v$. We claim that this contradicts $Q$ being hollow-maximal or empty-maximal.
For this, let $C_v=v + \R_{\ge 0}(Q-v)$ be the cone of $Q$ at $v$, then all the lattice points of $Q$ lie in the boundary of the cone. Let $u\in \text{int}(C_v)\cap \Z^d$ be such that $u$ is the only lattice point of $Q':=\conv(Q,u)$ in the interior of $C_v$. 
(Such a $u$ can be found, for example, minimizing in $\text{int}(C_v)\cap \Z^d$ any supporting linear functional of $C_v$). Then $Q'$ strictly contains $Q$ and it is still empty or hollow if $Q$ was empty or hollow, respectively (see Figure~\ref{fig:maximal_nonsimplex}).
\end{proof}

\begin{figure}[htb]
\centerline{\includegraphics{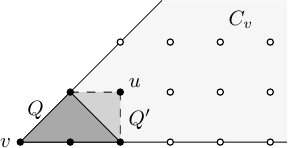}}
\caption{Finding vertices $u$ and $v$ not contained in a common facet in the proof of Lemma~\ref{lemma:max_hollow_or_empty}.}
\label{fig:maximal_nonsimplex}
\end{figure}

With this we can now prove that $w^\infty(d)$ is at least $\max\{w_E(d-1),w_H(d-2)\}$.

\begin{corollary}
\label{coro:infinite_nonsimplex}
For every $d\ge 3$ there exists an empty $(d-1)$-polytope of width $w_E(d-1)$ with infinitely many lifts of bounded size. 
In particular, $w^\infty(d) \ge w_E(d-1)$. 
\end{corollary}

\begin{proof}
Lemma~\ref{lemma:max_hollow_or_empty} implies that $w_E(d-1)$ is achieved by a non-simplex $Q$, and then Corollary~\ref{coro:holds-empty-nonsimplex} shows $Q$ has infinitely many lifts of bounded size. Corollary~\ref{coro:one_hollow_Q} implies then that $w^\infty(d)\ge\width(Q)=w_E(d-1)$.
\end{proof}

We call a polytope \emph{bipyramid}, if there are two vertices $u$ and $v$ such that every facet is a pyramid with apex either $u$ or $v$, and there is no facet containing both. Hollow bipyramids clearly satisfy the conditions of Lemma~\ref{lemma:infinite-hollow+(semi)empty}, hence they have infinitely many lifts of bounded size.

\begin{lemma}
\label{lemma:bipyramids}
For every $d\ge 2$ there exists a hollow bipyramid of dimension $d$ and width $w_H(d-1)$.
\end{lemma}

\begin{proof} 
By induction on $d$. For $d=2$, the unit square is a hollow bipyramid of width $1=w_H(1)$. For higher $d$,
let us first see that there exists a hollow $(d-1)$-polytope $Q$ of width $w_H(d-1)$ and having two lattice points $u$ and $v$ not sharing any facet. 
\begin{itemize}
\item If $w_H(d-1)=w_H(d-2)$ then let $Q$ be a hollow bipyramid of dimension $d-1$ and width $w_H(d-2)$, which exists by induction hypothesis.
\item If $w_H(d-1)>w_H(d-2)$ then there are finitely many hollow $(d-1)$-polytopes of width $w_H(d-1)$\cite{NillZiegler}, hence there is one such $Q$ that is maximal. By Lemma \ref{lemma:max_hollow_or_empty}, there are lattice points $u$ and $v$ in $Q$ not contained in the same facet.
\end{itemize}
Now consider the convex hull of $\left( Q\times \{0\} \right) \cup\{(u,h), (v,-h)\}$. This is a hollow bipyramid of dimension $d$ and, for sufficiently large $h$, it has the same width of $Q$ (by Theorem~\ref{theorem:same_width}).
\end{proof}

\begin{corollary}
\label{coro:infinite_bipyramid}
For every $d\ge 3$ there exists a hollow $(d-1)$-polytope of width $w_H(d-2)$ with infinitely many lifts of bounded size. 
In particular, $w^\infty(d) \ge w_H(d-2)$. 
\end{corollary}

\begin{proof}
Let $Q$ be a hollow $(d-1)$-dimensional bipyramid of width $w_H(d-2)$, which exists by Lemma~\ref{lemma:bipyramids}. 
Lemma~\ref{lemma:infinite-hollow+(semi)empty} shows $Q$ has infinitely many lifts of bounded size. 
Corollary~\ref{coro:one_hollow_Q} implies then that $w^\infty(d)\ge\width(Q)=w_H(d-2)$.
\end{proof}

This finally allows us to prove that:

\begin{theorem}
\label{thm:one_hollow_Q}
For all $d\ge 3$, 
$w^\infty(d)$ equals the maximum width of a lattice $(d-1)$-polytope $Q$ that admits infinitely many lifts of bounded size. Moreover, $Q$ is hollow.
\end{theorem}

\begin{proof}
That $w^\infty(d)$ is at least the width of any lattice $(d-1)$-polytope with infinitely many lifts of bounded size is Corollary~\ref{coro:one_hollow_Q}.

For the other inequality, Corollary~\ref{coro:infinite_bipyramid} proves the statement in the case when $w^\infty(d)=w_H(d-2)$ (it proves as well that $w^\infty(d) \ge w_H(d-2)$).

The only remaining case is then when $w^\infty(d) > w_H(d-2)$.
First of all, since $w^\infty(3)=1$ and by Proposition~\ref{proposition:monotone_d}, we have that $w^\infty(d)>0$ for all $d\ge 3$ (this guarantees the existence of infinitely many lattice $d$-polytopes of some fixed size).
Let $n$ be such that $W:=w^\infty(d) = w^\infty(d,n)$.
That is, there is an infinite family $\{P_i\}_{i\in \N}$ of lattice $d$-polytopes of size $n$ and width $W$.
Without loss of generality (Lemma~\ref{lemma:finite-nonprojecting}) assume all $P_i$'s are hollow and have a hollow lattice $(d-1)$-dimensional projection $Q_i$. 
Since projecting does not decrease the width, every $Q_i$ has width at least $W$, and since $W=w^\infty(d) > w_H(d-2)$ no $Q_i$ admits a hollow projection to dimension $d-2$.
This implies the family $\{Q_i\}_{i\in \N}$ to be finite, so one of them, call it $Q$, lifts to infinitely many members of the family $\{P_i\}_{i\in \N}$.
Theorem~\ref{theorem:same_width} implies then that $Q$ has width \emph{exactly} $W$.

Any $Q$ with infinitely many lifts of bounded size is hollow, by Lemma~\ref{lemma:finite-nonprojecting} and Corollary~\ref{coro:equiv_statements}.
\end{proof}

%%%%%%%%%%%%%%%%%%%%%%%%%%%%%%%%%%%
%%%%%%%				SECTION 4 			 %%%%%%%%%
%%%%%%%%%%%%%%%%%%%%%%%%%%%%%%%%%%%

\section{Polytopes with finitely many lifts of bounded size}
\label{sec:finitely-hollow-lifts}

\begin{lemma}
\label{lemma:finite-pyramid}
Let $Q$ be a lattice polytope that is a pyramid with base $F$. 
If $F$ has finitely many lifts of bounded size, then so does $Q$.
\end{lemma}

\begin{proof}
Let $Q\subset \R^{d-1}$ be lattice $(d-1)$-polytope that is a pyramid with base $F$ and apex $v$. 
Any tight lift of $Q$ is of the form $P(\tilde F,h):=\conv(\tilde F 
\cup \{\tilde v\})$, where $\tilde F$ is a tight lift of $F$ and
$\tilde v=(v,h)$ is a point in the fiber of $v$.
Since $\tilde F$ is contained in some hyperplane $H$ orthogonal to
$\{x_d=0\}$ and containing $F\times\{0\}$, $P(\tilde F,h)$ is a 
pyramid with base $\tilde F$ and apex $\tilde v$ (see Figure~\ref{fig:pyramid}).

Let $m$ be the distance from $v$ to
$F$, $P(\tilde F,h)$ is equivalent to $P(\tilde F,h+m)$ for
all $h\in \Z$ (we leave it to the reader to derive the unimodular
transformation). That is, there are at most $m$ values of $h$ that
give non-equivalent tight lifts $P(\tilde F,h)$, for any fixed
$\tilde F$. By hypothesis, there are only finitely
many such $\tilde F$ of bounded size, hence finitely many
tight lifts of bounded size of $Q$. Corollary~\ref{coro:equiv_statements} implies the statement.
\end{proof}
\begin{figure}[htb]
\centerline{\includegraphics{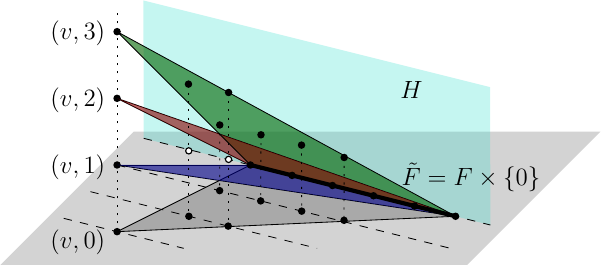}}
\caption{The setting of the proof of Lemma~\ref{lemma:finite-pyramid}. In the figure, the case when $\tilde F$ is the tight lift $F \times \{0\}$ is represented. The apex $v$ is at distance $3$, hence $(v,3)$ yields equivalent lift as $(v,0)$, while $(v,1)$ and $(v,2)$ do not.}
\label{fig:pyramid}
\end{figure}

\begin{corollary}
\label{coro:finite_simplex}
Lattice simplices have only finitely many
lifts of bounded size.
\end{corollary}

\begin{proof}
Using induction on the dimension and Lemma~\ref{lemma:finite-pyramid},
this follows from the fact that a single lattice point has only finitely many lifts of bounded size.
\end{proof}

We now want to show that non-hollow lattice polytopes have only finitely many lifts of bounded size. The following geometric lemma (in which $Q$ need not be a lattice polytope) will be helpful.

\begin{lemma}
\label{lemma:volume_non_hollow}
Let $\pi:\R^d\to \R^{d-1}$ be the standard projection that forgets the last coordinate.
Let $q$ be a point in the interior of a $(d-1)$-polytope $Q$. Then, there is a $c\in \R$ such that for every $d$-polytope $P\subset \R^d$ with $\pi (P)=Q$ we have
\[
{\operatorname{vol} (P)} \le c\cdot  {\operatorname{length} (P\cap \pi^{-1} (q))}.
\]
\end{lemma}

\begin{proof}
Assume without loss of generality that $q$ is the origin and that the vertical segment $P\cap \pi^{-1} (q)$ goes from $(q,0)$ to $(q,1)$. 
This is no loss of generality since the parameter ${{\operatorname{vol} (P)/\operatorname{length} (P\cap \pi^{-1} (q))}}$ does not change by vertical translation or vertical dilation/contraction of $P$. Notice that the polytope $P$ may be rational.
Under these assumptions what we want to show that there is a global upper bound $c$ for the volume of $P$.

By considering respective supporting hyperplanes of $P$ at $(q,0)$ and $(q,1)$ we see that $P$ is contained in the region $f_1(x_1,\dots,x_{d-1}) \le x_d \le f_2(x_1,\dots,x_{d-1}) +1$, for some linear functionals $f_1,f_2\in (\R^{d-1})^*$, and there is no loss of generality in assuming that $P$ actually equals 
the intersection of $\pi^{-1}(Q)$ with that region (see Figure~\ref{fig:slices}). 
Now, for $\pi(P)$ to equal $Q$ we need $f_1 - f_2 \le 1$ on $Q$, which is equivalent to saying that $f_1-f_2$ is in the polar $Q^\vee$ of $Q$. 
The volume of $P$ is a continuous function of the functional $f_1-f_2$. (In fact, it equals the integral in $Q$ of the function $1 + f_2 -f_1$). 
Since the origin is in the interior of $Q$, $Q^\vee$ is compact, and there is a global bound on the volume of $P$.
\begin{figure}[htb]
\centerline{\includegraphics{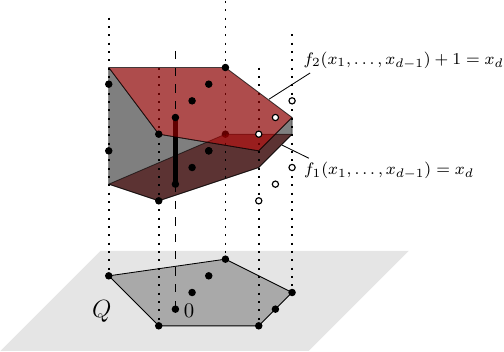}}
\caption{The setting of the proof of Lemma~\ref{lemma:volume_non_hollow}. The figure shows the rational $d$-polytope $\pi^{-1}(Q) \cap \{f_1(x_1,\dots,x_{d-1}) \le x_d \le f_2(x_1,\dots,x_{d-1}) +1\}$.}
\label{fig:slices}
\end{figure}
\end{proof}

\begin{corollary}
\label{coro:finite-polytope-interiorpoint}
A non-hollow lattice polytope has only finitely many lifts of bounded size.
\end{corollary}

\begin{proof}
Let $Q\subset\R^{d-1}$ be a lattice $(d-1)$-polytope and let $q\in \Z^{d-1}$ be an interior lattice point of $Q$.
A bound $n$ for the size of a lift $P$ of $Q$ implies a bound $n+1$ for the length of $\pi^{-1}(q) \cap P$. By Lemma~\ref{lemma:volume_non_hollow}, this gives a bound for the volume of $P$. Since there are only finitely many lattice $d$-polytopes with bounded volume (Hensley~\cite[Thm.~3.6]{Hensley}), the result follows.
\end{proof}

%%%%%%%%%%%%%%%%%%%%%%%%%%%%%%%%%%%
%%%%%%%				SECTION 5 			 %%%%%%%%%
%%%%%%%%%%%%%%%%%%%%%%%%%%%%%%%%%%%

\section{The finiteness threshold width in dimension 4}
\label{sec:dim4}

According to Theorem~\ref{thm:one_hollow_Q}, $w^\infty(4)$ 
equals the largest width of a hollow lattice $3$-polytope with 
infinitely many lifts of bounded size. Since $w^\infty(4)\ge 2$ 
is known (Haase and Ziegler~\cite[Proposition 6]{HaaseZiegler} 
showed infinitely many empty $4$-simplices of width two), we only need to 
look at hollow $3$-polytopes of width at least $3$. 
Let us show that there are only five of them, all of width three
(see Lemma~\ref{lemma:w_H(3)} and Figure~\ref{fig:max-hollow}).

We start with the following classification of hollow lattice $3$-polytopes:

\begin{theorem}[~\protect{\cite[Theorem 1.3]{Treutlein}}]
\label{thm:hollow_dim3}
Any hollow lattice $3$-polytope falls exactly under one of the
following categories:
\begin{enumerate}
\item It has width $1$. All polytopes of width $1$ are hollow and
 there are infinitely many of them for each size.

\item It has width $2$ and admits a projection onto the polygon
 $2\Delta_2$. There are infinitely of them, although finitely many
 for each fixed size.

\item It has width $\ge 2$, and does not admit a projection to
 $2\Delta_2$. There are finitely many of them, regardless the
 size. They are all contained in hollow-maximal $3$-polytopes.
\end{enumerate}
\end{theorem}

The hollow-maximal $3$-polytopes referred to in part (3) have been
enumerated in~\cite{AverkovWagnerWeismantel,AKW15}. More precisely,
Averkov, Wagner and Weismantel~\cite{AverkovWagnerWeismantel}
classified the hollow lattice $3$-polytopes that are not properly
contained in any convex body without interior lattice points. Then Averkov,
Kr\"umpelmann and Weltge~\cite{AKW15} showed that the maximal lattice
$3$-polytopes in this sense (which they call $\R$-maximal) coincide
with the hollow-maximal lattice $3$-polytopes in our sense (which they call
$\Z$-maximal). It is known that these two notions of maximality for hollow
polytopes do not coincide in dimensions four and
higher~\cite{NillZiegler}.

\begin{theorem}[~\protect{\cite[Theorem 2.2]{AverkovWagnerWeismantel}}
 and~\protect{\cite[Theorem 1]{AKW15}}]
\label{theorem:averkov_et_al}
There are the following $12$ hollow-maximal lattice $3$-polytopes: 
\[
\arraycolsep=4pt
\begin{array}{ccc}
\mathcal{M}_1\left( \begin{array}{cccc}
 0 & 2 & 0 &0 \\ 
 0 & 0 & 3 & 0\\ 
 0 & 0 & 0 & 6\end{array} \right)
&
\mathcal{M}_2\left( \begin{array}{cccc}
 0 & 2 & 0 &0 \\ 
 0 & 0 & 4 & 0\\ 
 0 & 0 & 0 & 4\end{array} \right)
&
\mathcal{M}_3\left( \begin{array}{cccc}
 0 & 3 & 0 &0 \\ 
 0 & 0 & 3 & 0\\ 
 0 & 0 & 0 & 3\end{array} \right)
\\
\\

\mathcal{M}_4\left( \begin{array}{cccc}
 0 & 1 & 2 &3 \\ 
 0 & 0 & 4 & 0\\ 
 0 & 0 & 0 & 4\end{array} \right)
&
\mathcal{M}_5\left( \begin{array}{cccc}
 0 & 1 & 2 &3 \\ 
 0 & 0 & 5 & 0\\ 
 0 & 0 & 0 & 5\end{array} \right)
&
\mathcal{M}_6\left( \begin{array}{cccc}
 0 & 3 & 1 &2 \\ 
 0 & 0 & 3 & 0\\ 
 0 & 0 & 0 & 3\end{array} \right)
\\
\end{array}
\]
\[
\arraycolsep=4pt
\begin{array}{ccc}

\mathcal{M}_7\left( \begin{array}{cccc}
 0 & 4 & 1 &2 \\ 
 0 & 0 & 2 & 0\\ 
 0 & 0 & 0 & 4\end{array} \right)
&
\mathcal{M}_8\left( \begin{array}{ccccc}
 2 &-2 & 0 & 0& 1\\ 
 0 & 0 & 2 &-2& 1\\ 
 0 & 0 & 0 & 0& 2\end{array} \right)
&
\mathcal{M}_9\left( \begin{array}{ccccc}
-1 & 2 & 0 & 0& 1\\ 
 0 & 0 &-1 & 2& 1\\ 
 0 & 0 & 0 & 0& 3\end{array} \right)
\end{array}
\]
\[
\mathcal{M}_{10}\left( \begin{array}{cccccc}
 1 & 0 &-1 & 2& 1& 0 \\ 
 0 & 1 &-1 & 2& 3& 1\\ 
 0 & 0 & 0 & 3& 3& 3\end{array} \right)
\qquad\qquad
\mathcal{M}_{11}\left( \begin{array}{cccccc}
 1 &-1 & 0 & 2& 0& 1 \\ 
 0 & 0 & 2 & 0& 0& 2\\ 
 0 & 0 & 0 & 2& 2& 2\end{array} \right)
\]
\[
\arraycolsep=4pt
\mathcal{M}_{12}\left( \begin{array}{cccccccc}
0 &-1 & 1 & 0 & 1 & 0 & 2& 1\\ 
0 & 1 & 1 & 2 & 1 & 2 & 2& 3\\ 
0 & 0 & 0 & 0 & 2 & 2 & 2& 2\end{array} \right)
\]
They all have width two except 
$\mathcal{M}_{3}$, $\mathcal{M}_{5}$, $\mathcal{M}_{6}$, $\mathcal{M}_{9}$ and $\mathcal{M}_{10}$,
of width three.
\end{theorem}

In particular, every hollow $3$-polytope has width $\le 3$ and those of width three are contained in one of $\mathcal{M}_{3}$, $\mathcal{M}_{5}$, $\mathcal{M}_{6}$, $\mathcal{M}_{9}$ and $\mathcal{M}_{10}$. 
These five polytopes are pictured in Figure~\ref{fig:max-hollow}, taken from~\cite{AKW15}. (The coordinate system in the figure is not the same as in the definition)

\begin{figure}[htb]
\includegraphics[scale=.9]{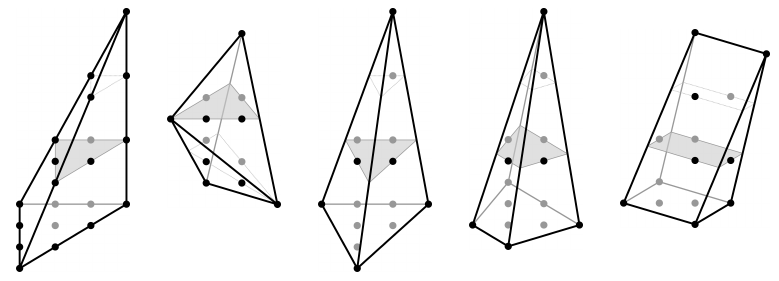}
\centerline{$
\mathcal{M}_{3} \qquad\qquad\quad
\mathcal{M}_{5} \qquad\qquad\quad
\mathcal{M}_{6} \qquad\qquad\quad
\mathcal{M}_{9} \qquad\qquad\quad
\mathcal{M}_{10}
$}
\caption{The five hollow $3$-polytopes of width three. This picture has been taken from Averkov et al~\cite{AKW15}.}
\label{fig:max-hollow}
\end{figure}

A priori there could be proper subpolytopes of one of these five that still have width three, but it is not difficult to prove that this is not the case:

\begin{lemma}
\label{lemma:w_H(3)}
The only lattice hollow $3$-polytopes of width $> 2$ are $\mathcal{M}_{3}$, $\mathcal{M}_{5}$, $\mathcal{M}_{6}$, $\mathcal{M}_{9}$ and $\mathcal{M}_{10}$, and they have width three.
\end{lemma}

\begin{proof}
It suffices to check that all the subpolytopes of $\mathcal{M}_{3}$, $\mathcal{M}_{5}$, $\mathcal{M}_{6}$, $\mathcal{M}_{9}$ and $\mathcal{M}_{10}$ obtained by removing a single vertex have width two (or lower). For this, in turn, it suffices to find for each of the five polytopes and each vertex of it, an integer affine functional having value $3$ on that vertex and values $0$, $1$ or $2$ in all the others. Such functionals are specified in the following matrices
$\mathcal{F}_3$, $\mathcal{F}_5$,
$\mathcal{F}_6$, $\mathcal{F}_9$, and $\mathcal{F}_{10}$,
where the $i$-th row of matrix $\mathcal{F}_j$ is the functional corresponding to the vertex that is the $i$-th column of the matrix $\mathcal{M}_j$ from Theorem~\ref{theorem:averkov_et_al}. A row $(a\ b\ c\ |\ d)$ represents the functional $(x,y,z)\mapsto ax+by+cz+d$:
\[
\arraycolsep=4pt
\mathcal{F}_3\left( \begin{array}{ccc|c}
-1 &-1 &-1 & 3 \\ 
 1 & 0 & 0 & 0 \\ 
 0 & 1 & 0 & 0\\ 
 0 & 0 & 1 & 0\end{array} \right)
\quad
\mathcal{F}_5\left( \begin{array}{ccc|c}
-1 & 0 & 0 & 3 \\
 2 & -1 & -1 & 1\\ 
-2 &  1 &  1 & 2\\
1 & 0 & 0 & 0 
\end{array} \right)
\quad
\mathcal{F}_6\left( \begin{array}{ccc|c}
-1 & 0 & 0 &3 \\ 
 1 & 0 & 0 &0 \\ 
 0 & 1 & 0 & 0\\ 
 0 & 0 & 1 & 0\end{array} \right)
\]
\[
\arraycolsep=4pt
\mathcal{F}_9\left( \begin{array}{ccc|c}
-1 & 0 & 0 &2 \\ 
 1 & 0 & 0 &1 \\ 
 0 & -1 & 0 & 2\\ 
 0 & 1 & 0 & 1\\ 
 0 & 0 & 1 & 0\end{array} \right)
\qquad
\mathcal{F}_{10}\left( \begin{array}{ccc|c}
 1 &-1 & 0 & 2\\ 
 0 & 1 & -1 & 2\\
-1 & 0 & 0 &2 \\ 
 1 & 0 & 0 &1 \\ 
-1 & 1 & 0 & 1\\ 
 0 &-1 & 1 & 1\end{array} \right)
\]

In the two that are perhaps less obvious, $\mathcal{M}_{5}$ and $\mathcal{M}_{10}$, the (linear parts of) the functionals come in pairs of opposite ones. Figure~\ref{fig:sub-hollow} shows projections along which these functionals are coordinates (one picture, with two coordinate functionals, for $\mathcal{M}_{5}$, three pictures with the horizontal coordinate in each picture as one of the functionals, for $\mathcal{M}_{10}$).
\begin{figure}[h]
\includegraphics[scale=.48]{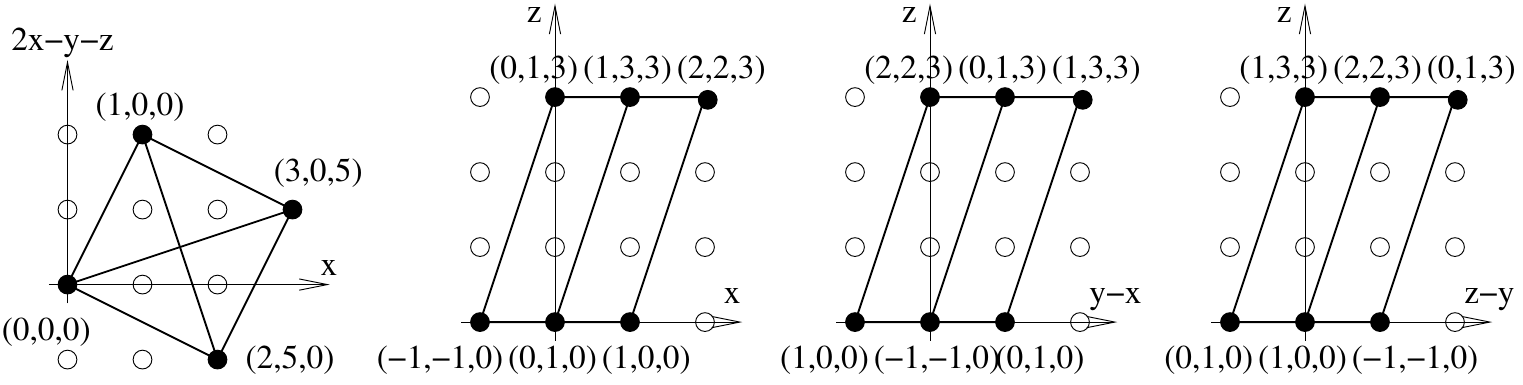}
\centerline{
$
\mathcal{M}_{5} 
\qquad\qquad\quad
\qquad\qquad\quad
\qquad\qquad\quad
\qquad
\mathcal{M}_{10}
\qquad\quad
\qquad\qquad\quad
$
}
\caption{Projections showing that all proper subpolytopes of $\mathcal{M}_{5}$ (left picture) and $\mathcal{M}_{10}$ (three right pictures) have width $<3$.}
\label{fig:sub-hollow}
\end{figure}
\end{proof}

\begin{remark}
\label{rem.data}
As a double-check we have enumerated, using Polymake~\cite{polymake}, all subpolytopes of $\mathcal{M}_1,\ldots,\mathcal{M}_{12}$ of width
$\geq 2$, ordered by size. Our width algorithm is included in releases of Polymake starting with
version 3.0 as a property of a polytope with command \texttt{LATTICE\_WIDTH}.
The lists of the subpolytopes and the algorithms we used to compute them can be found on \url{http://ehrhart.math.fu-berlin.de/Research/Data/}, and also as ancillary files to \href{http://arxiv.org/abs/1607.00798v3}{\tt arXiv:1607.00798v3}.
\end{remark}

\begin{corollary}[Finiteness Threshold Width in dimension $4$]
\label{coro:main}
$w^\infty(4)=2$. 
That is, for each $n\ge 5$, there exist only finitely many lattice $4$-polytopes of size $n$ and width larger than two.
\end{corollary}

\begin{proof}
That $w^\infty(4)\ge2$ follows from Example~\ref{exm:threshold_via_lifts}. 

In the light of Theorem~\ref{thm:one_hollow_Q}, in order to prove $w^\infty(4)\le2$ we only need to check that no hollow $3$-polytope of width larger than two has infinitely many lifts of bounded size.
Lemma~\ref{lemma:w_H(3)} tells us that there are only five polytopes to check, depicted in Figure~\ref{fig:max-hollow}.
$\mathcal{M}_{3}$, $\mathcal{M}_{5}$ and $\mathcal{M}_{6}$ are simplices and hence have only finitely many lifts of bounded size by Corollary~\ref{coro:finite_simplex}. 
That $\mathcal{M}_{9}$ and $\mathcal{M}_{10}$ have only finitely many lifts of bounded size is proved in Propositions~\ref{prop:M_9} and~\ref{prop:M_10} below. 
\end{proof}

\begin{proposition}
\label{prop:M_9}
The pyramid $\mathcal{M}_{9}$ has finitely many lifts of bounded size.
\end{proposition}

\begin{proof}
The base of the pyramid is a quadrilateral with three (relative) interior points. 
This quadrilateral has a finite number of
lifts of bounded size by Corollary~\ref{coro:finite-polytope-interiorpoint},
and the whole pyramid by Lemma~\ref{lemma:finite-pyramid}.
\end{proof}

\begin{proposition}
\label{prop:M_10}
The prism $\mathcal{M}_{10}$ has finitely many lifts of bounded size.
\end{proposition}

\begin{proof}
Let $u,v,w, u', v'$ and $w'$ be the vertices of the prism, where
$uu',vv',ww'$ are edges.
Let $Q:=\conv\{u,v,w,u',v'\}\subset \mathcal{M}_{10}$. It is a
quadrangular pyramid over a polygon with interior points.

Any tight lift of $\mathcal{M}_{10}$ will be of the form $P(\tilde
Q,\tilde w ')=\conv(\tilde Q \cup \{\tilde w'\}) $, where $\tilde Q$ is a tight lift of $Q$ and
$\tilde w '$ is a point in the fiber of $w'$. By
Lemma~\ref{lemma:finite-pyramid}
and Corollary~\ref{coro:finite-polytope-interiorpoint}, there are only finitely many
such $\tilde Q$ of bounded size. Fix one, and let us see that there
are only finitely many possibilities for $\tilde w '$.

Each lift $\tilde w '$ (together with the fixed tight lift $\tilde Q$)
induces a lift of the quadrilateral $R:=\conv\{u,
w, u', w'\}$. We claim that at most two choices of ${\tilde w'}$ correspond to equivalent lifts of $R$.

By fixing $\tilde Q$ we already have fixed a lift of the three
vertices $u, w, u'$. These three lifts are contained in a plane
$\Pi$. On the other hand, the possible lifts of the point $w'$ are in
the line $\pi^{-1}(w')$.
This line is not contained in $\Pi$, so these tight lifts of $R$
are all $3$-dimensional (except for at most one lift of $\tilde w'$), and their
volume is proportional to the distance between ${\tilde w'}$ and
$\Pi$. That is, each of the possibilities for ${\tilde w'}$ induces
non-equivalent tight lifts of the quadrilateral $R$, up to (perhaps)
reflection with respect to the plane $\Pi$.

Now, as the quadrilateral $R$ contains interior points, Corollary~\ref{coro:finite-polytope-interiorpoint}
implies that it has only finitely many lifts of bounded size. Infinitely many choices of $\tilde w'$ would then have unbounded size, 
and so would happen for $P(\tilde Q,\tilde w')$. That is, $\mathcal{M}_{10}$ has only finitely many tight lifts of bounded size, and Corollary~\ref{coro:equiv_statements} implies the statement.
\end{proof}

\bibliographystyle{amsalpha}

\bibliography{finiteness}

\end{document}